\documentclass[a4paper]{article}
\usepackage[utf8]{inputenc}
\usepackage{amsmath}
\usepackage{amssymb}
\usepackage{amsthm}
\usepackage{verbatim}
\usepackage{enumerate}
\usepackage{amsfonts}
\usepackage{amsmath}
\usepackage{hyperref}
\usepackage{comment}
\usepackage{authblk}

\usepackage{fullpage}
\usepackage{setspace}
\newtheorem{theorem}{Theorem}
\newtheorem{conj}[theorem]{Conjecture}
\newtheorem{lemma}[theorem]{Lemma}

\newtheorem*{remark*}{Remark}
\newtheorem*{prob*}{Problem}

\def\bpmod#1{\nonscript\mskip-\medmuskip\mkern5mu\mathbin
  {\rm(mod }\penalty900\mkern
  5mu\nonscript\mskip-\medmuskip#1\rm)}

\title{The mod $k$ chromatic index of graphs is $O(k)$%
\thanks{\scriptsize This research has been partially supported by
  Coordenação de Aperfeiçoamento  
de Pessoal de Nível Superior - Brasil -- CAPES -- Finance Code 001. 
F.~Botler is supported by CNPq (423395/2018-1)
and by FAPERJ (211.305/2019). Y.~Kohayakawa is partially supported by
CNPq (311412/2018-1, 423833/2018-9) and FAPESP (2018/04876-1).
The research that led to this paper started at WoPOCA 2019, which was
financed by FAPESP (2015/11937-9) and CNPq (425340/2016-3,
423833/2018-9). FAPERJ is the Rio de Janeiro Research Foundation.
FAPESP is the S\~ao Paulo Research Foundation.  CNPq is the National 
Council for Scientific and Technological Development of Brazil. }
}
\author{Fábio Botler,\footnote{\scriptsize Programa de Engenharia de
    Sistemas e Computação, COPPE, Universidade Federal do Rio de
    Janeiro, Brazil} 
  $\;$Lucas Colucci and Yoshiharu
  Kohayakawa\footnote{\scriptsize
    Instituto de Matemática e Estatística, Universidade de São Paulo,
    Brazil}}

\usepackage{datetime}
\usepackage{lineno}
\newcommand*\patchAmsMathEnvironmentForLineno[1]{%
\expandafter\let\csname old#1\expandafter\endcsname\csname #1\endcsname
\expandafter\let\csname oldend#1\expandafter\endcsname\csname end#1\endcsname
\renewenvironment{#1}%
{\linenomath\csname old#1\endcsname}%
{\csname oldend#1\endcsname\endlinenomath}}%
\newcommand*\patchBothAmsMathEnvironmentsForLineno[1]{%
\patchAmsMathEnvironmentForLineno{#1}%
\patchAmsMathEnvironmentForLineno{#1*}}%
\AtBeginDocument{%
\patchBothAmsMathEnvironmentsForLineno{equation}%
\patchBothAmsMathEnvironmentsForLineno{align}%
\patchBothAmsMathEnvironmentsForLineno{flalign}%
\patchBothAmsMathEnvironmentsForLineno{alignat}%
\patchBothAmsMathEnvironmentsForLineno{gather}%
\patchBothAmsMathEnvironmentsForLineno{multline}%
}

\begin{document}
\onehalfspace
\shortdate
\yyyymmdddate
\settimeformat{ampmtime}
\date{}

\maketitle

\begin{abstract}
  Let $\chi'_k(G)$ denote the minimum number of colors needed to color
  the edges of a graph $G$ in a way that the subgraph spanned by the
  edges of each color has all degrees congruent to $1 \bpmod k$.  Scott
  [{\em Discrete Math.~175}, 1-3 (1997), 289--291] proved that
  $\chi'_k(G)\leq5k^2\log k$, and thus settled a question of Pyber
  [{\em Sets, graphs and numbers} (1992), pp.~583--610], who had asked
  whether~$\chi_k'(G)$ can be bounded solely as a function of~$k$.  We
  prove that $\chi'_k(G)=O(k)$, answering affirmatively a question of
  Scott.
\end{abstract}

\vspace{.3cm}

Throughout this paper, unless stated otherwise, \(k\geq 2\) denotes an integer. 
All graphs considered here are simple, and $e(G)$ denotes the number
of edges in the graph $G$.
A \emph{\(\chi'_k\)-coloring} of \(G\) is a coloring of the edges of
\(G\) in which the subgraph spanned by the edges of each color has all
degrees congruent to $1 \bpmod k$, 
and we denote by $\chi'_k(G)$ the minimum number of colors in a
\(\chi'_k\)-coloring of \(G\). Pyber~\cite{pyber1991covering} proved
that $\chi'_2(G) \leq 4$ for every graph $G$ and asked whether
$\chi'_k(G)$ is bounded by some function of \(k\) only. Scott
\cite{scott1997graph} proved that $\chi'_k(G) \leq 5k^2\log k$ for
every graph \(G\), and in turn asked if $\chi'_k(G)$ is in fact
bounded by a linear function of~$k$.  
This would be best possible apart from the multiplicative constant,
as \(\chi'_k(K_{1,k})=k\).  In this paper, we answer Scott's question
affirmatively. 

We shall make use of the following two results.

\begin{lemma}[Mader \cite{mader1978existenz}]\label{lem:connectivity}
  If $k \geq 1$, $G$ is a graph on $n$ vertices, and $e(G) \geq 2kn$, then $G$
  contains a $k$-connected subgraph.  
\end{lemma}

\begin{lemma}[Thomassen \cite{thomassen2014graph}]\label{lem:factormodk}
  If $k \geq 1$ and $G$ is a $(12k-7)$-edge-connected graph with an even number of
  vertices, then $G$ has a spanning subgraph in which each vertex
  has degree congruent to $k \bpmod {2k}$.
\end{lemma}

We say that a graph \(G\) is \emph{\(k\)-divisible} if \(k\) divides
the degree
of every vertex of~$G$.
Lemma~\ref{lem:factormodk} thus guarantees the existence of a
\(k\)-divisible spanning subgraph in~$G$ when \(G\) is
\((12k-7)\)-edge-connected.

\begin{lemma}\label{lem:main}
  If \(G\) is a graph on $n$ vertices and does not contain a non-empty
  \(k\)-divisible subgraph, then \(e(G) < 2(12k-6)n\). 
\end{lemma}
\begin{proof}
  Let $G$ be a graph and suppose that $e(G) \geq 2(12k-6)n$.
  Lemma~\ref{lem:connectivity} tells us that~$G$ contains a
  $(12k-6)$-connected subgraph~$H$.  If~$H$ has an odd number of
  vertices, let $H'=H-v$ for an arbitrary vertex \(v\in V(H)\).
  Otherwise, let $H'=H$.  Then~$H'$ has an even number of vertices and
  is $(12k-7)$-connected, and hence is $(12k-7)$-edge-connected also.
  Lemma \ref{lem:factormodk} tells us that $H'$ contains a non-empty
  $k$-divisible subgraph and therefore so does~$G$.
\end{proof}

Given an integer \(d\), we say that a graph \(G\) is
\emph{\(d\)-degenerate} if there is an ordering \(v_1,\ldots,v_n\) of
its vertices for which the number of neighbors of \(v_i\) in
\(V_i=\{v_1,\ldots,v_{i-1}\}\) is at most \(d\).  In this case, the
neighbors of~$v_i$ in~$V_i$ are its \emph{left neighbors} and the
neighbors of~$v_i$ in~$\{v_{i+1},\dots,v_n\}$ are its \emph{right
  neighbors}.  An edge \(uv\) is a \emph{left edge} of \(v\) if \(u\)
is a left neighbor of \(v\).  A left edge~$uv$ of~$v$ is a \emph{right
edge} of~$u$.

\begin{lemma}\label{lem:degen}
  Let $G$ be a $d$-degenerate graph.  Then $\chi'_k(G) \leq 4d+2k-2$.
\end{lemma}
\begin{proof}
  Let $V(G) = \{v_1,\dots,v_n\}$ be an ordering of \(V(G)\) as above.
  In what follows, we color $G$ by coloring the right edges of $v_i$
  for each $i \in \{1,\dots,n-1\}$ in turn, so that at each step we
  have a \(\chi'_k\)-coloring of the graph spanned by the right edges
  of \(v_1,\ldots,v_i\) with the following properties: for
  each~$1\leq j\leq n$,
  each of the colored, left edges of $v_j$ is colored with a distinct
  color, and the colors used on the right edges of $v_j$ are distinct
  from the colors used on its left edges.
  We proceed by induction on $i$. Let \(i\in\{1,\ldots,n-1\}\), and
  suppose that, for every \(j<i\), the right edges of $v_j$ is
  colored as above.  This implies that all the left edges of \(v_i\)
  are colored and no right edge of \(v_i\) is colored.  In what
  follows, we color the set \(S\) of right edges of \(v_i\) while
  keeping the properties of the partial coloring.

  Since \(G\) is \(d\)-degenerate, \(v_i\) has at most \(d\) left
  edges, and hence we have at least $3d+2k-2$ colors to use on its
  right edges.  We partition these colors arbitrarily into sets \(A\)
  and \(B\) so that~$|A|=d+k$ and~$|B|\geq2d+k-2$.  Let \(j>i\) and
  suppose \(v_j\) is a (right) neighbor of \(v_i\).  Note that at most
  \(d-1\) left edges of \(v_j\) are colored.  We say that a color
  \(c\) is \emph{forbidden} at \(v_j\) if a left edge of \(v_j\) is
  colored with \(c\), and we call the colors in~$A$ that are not
  forbidden at~$v_j$ \emph{available} at~$v_j$.

  Let \(S^*\) be a maximal subset of \(S\) that can be colored with
  colors in \(A\) in a way that (a) each right edge \(v_iv_j\in S^*\)
  is colored with a color available at \(v_j\), and (b) the number of
  edges in \(S^*\) colored with any given color is congruent to
  \(1\bpmod{k}\).
  Let \(\bar{S} = S\setminus S^*\) be the set of the remaining
  edges in~\(S\).  We claim that \(|\bar{S}| < |A|\).
  Assume for a contradiction that~$|\bar S|\geq|A|$.  For each
  edge \(e=v_iv_j\in \bar{S}\), let \(A_e\) be the set of colors
  available at \(v_j\), and for each color \(x\in A\), let
  \(\bar{S}_x\) be the set of edges \(e\) in \(\bar{S}\) for which
  \(x\in A_e\).  Note that
  \(\sum_{e\in \bar{S}} |A_e| = \sum_{x\in A} |\bar{S}_x|\)
  and that $|A_e|\geq|A|-(d-1)=d+k-d+1=k+1$ for every~$e\in\bar S$.
  Therefore
  \[(k+1)|A|\leq(k+1)|\bar{S}|\leq\sum_{e\in \bar{S}} |A_e| =
    \sum_{x\in A} |\bar{S}_x|\leq |A|\max\{|\bar{S}_x|\colon x\in
    A\},\] whence \(\max\{|\bar{S}_x|\colon x\in A\}\geq k+1\).  Let
  \(z\in A\) be such that
  \(|\bar{S}_z|=\max\{|\bar{S}_x|\colon c\in A\}\).  If some edge
  in~\(S^*\) is colored with \(z\), then we color \(k\) edges in~$\bar
  S$ with color \(z\).  If no edge in \(S^*\) is colored with \(z\), then
  we color \(k+1\) edges in~$\bar S$ with color \(z\).  In both cases
  we obtain a contradiction to the maximality of \(S^*\).  This shows
  that, indeed, $|\bar S|<|A|=d+k$.

  Finally, we color the edges of \(\bar{S}\) consecutively and with
  distinct colors in $B$.  This is possible, since for
  each~$v_iw\in\bar S$ there are at most
  $d-1+|\bar{S}|-1 \leq 2d+k-3 < |B|$ colors of~$B$ that are forbidden
  (the colors forbidden at~$w$ plus the colors of $B$ used on previous
  edges of $\bar{S}$).
\end{proof}

Our main result is a consequence of Lemmas~\ref{lem:main}
and~\ref{lem:degen}.

\begin{theorem}\label{thm:main}
  For every graph \(G\) we have $\chi'_k(G) \leq 198k-101$.
\end{theorem}
\begin{proof}
  Let $H$ be a maximal subgraph of $G$ for which
  \(\deg_H(v) \equiv 1 \pmod{k}\) for every \(v\in V(H)\), and let
  \(G' = G\setminus E(H)\).
  The maximality of \(H\) implies that $V(G)\setminus V(H)$ is
  independent, and that every vertex in $V(H)$ has at most $k-1$
  neighbors in $V(G)\setminus V(H)$.  Moreover, $G'[V(H)]$ has no
  non-empty \(k\)-divisible subgraph.  By Lemma \ref{lem:main}, every
  $J \subseteq G'[V(H)]$ has less than $2(12k-6)|V(J)|$
  edges, and hence its minimum degree is less than $48k-24$.
  This implies that every subgraph of $G'$ has a vertex of degree at
  most $49k-25$, and hence $G'$ is $(49k-25)$-degenerate.
  Lemma~\ref{lem:degen} tells us that~$G'$ has a $\chi_k'$-coloring
  with at most~$198k-102$ colors.  We then color \(E(H)\) with a new
  color, and the result follows.
\end{proof}

Alon, Friedland and Kalai~\cite{alon1984regular} proved that when
\(k\) is a prime power, the bound $2(12k-6)n$ in Lemma~\ref{lem:main}
can be replaced by $(k-1)n+1$, and conjectured that this holds for
every positive integer \(k\).  This result, together with Lemma
\ref{lem:degen}, implies that $\chi'_k(G) \leq 14k-9$ for any~$G$ and
any prime power~$k$.

Our arguments can be tweaked to give slightly better multiplicative
constants, but we do not think it is worth pursuing this because we
think we would be far from the truth still.  Although we are not able
to offer any strong evidence, we conjecture the following.

\begin{conj}
  There is a constant $C$ such that $\chi'_k(G) \leq k+C$ for every
  graph~$G$.
\end{conj}

We know that~$C$ in the conjecture above has to be at least~$2$,
because one can prove that the graph~$G$ obtained from a $K_{k,k}$ by
adding a universal vertex satisfies $\chi'_k(G) = k+2$.

\bibliographystyle{acm}
\bibliography{ref}

\endgroup 
\end{document}